\documentclass[10pt]{article}%
\usepackage{amsmath}
\usepackage{amsthm}
\usepackage{amssymb}
\usepackage{multicol}
\usepackage[active]{srcltx}
\usepackage{algorithm}
\usepackage{array}
\usepackage[dvipsnames,svgnames]{xcolor}
\usepackage{url}
\usepackage{tikz}
\usepackage{subfig}
\usepackage{graphicx}
\usepackage{pgfplots}
\usepackage{xcolor}
\usepackage{xcolor}
\usepackage{ulem}
\usepackage{comment}
\usepackage{amsfonts}%
\providecommand{\U}[1]{\protect \rule{.1in}{.1in}}
\newtheorem{theorem}{Theorem}

\newtheorem{corollary}[theorem]{Corollary}

\newtheorem{example}[theorem]{Example}

\newtheorem{lemma}[theorem]{Lemma}

\newtheorem{remark}[theorem]{Remark}

\begin{document}

\title{Characterizations of Strongly Quasiconvex Functions}
\author{Nicolas Hadjisavvas\thanks{Department of Product and Systems Design
Engineering, University of the Aegean, Hermoupolis, Syros, Greece. E-mail:
nhad@aegean.gr, ORCID-ID: 0000-0002-9895-8190}
\and Felipe Lara\thanks{Instituto de Alta investigaci\'on (IAI), Universidad de
Tarapac\'a, Arica, Chile. E-mail: felipelaraobreque@gmail.com;
flarao@academicos.uta.cl. Web: felipelara.cl, ORCID-ID: 0000-0002-9965-0921}}
\maketitle

\begin{abstract}
\noindent
We provide new necessary and sufficient conditons for ensuring strong
quasiconvexity in the nonsmooth case and, as a consequence, we provide a
proof for the differentiable case. Furthermore, we improve the quadratic
growth property for strongly quasiconvex functions.

\medskip

\noindent{\small \emph{Keywords}: Generalized convexity; Strong
quasiconvexity}

\medskip

\noindent \textbf{Mathematics Subject Classification:} 90C26; 90C30.

\end{abstract}



\section{Introduction}

Let $X$ be a normed space, $C\subseteq X$ be an open convex set and
$h:C\rightarrow \mathbb{R}$ be a function. The famous Arrow-Enthoven
characterization \cite{AE} says that: A diffe\-ren\-tia\-ble function
$h:C\rightarrow \mathbb{R}$ is quasiconvex if and only if for every $x,y\in C$,
we have
\begin{equation}
h(x)\leq h(y)~\Longrightarrow~\langle \nabla h(y),y-x\rangle \geq0.
\label{char:AE}%
\end{equation}

The previous characterization is very useful in generalized convexity and
monotonicity theory, continuous optimization and variational inequalities as
well as in its applications in economics and engineering among others (see
\cite{ADSZ,CamMar,HKS} for a comprehensive presentation).

A version of the previous result, adapted to strongly quasiconvex functions,
was given in \cite[Theorems 1 and 6]{VNC-2} in two parts. The first part is
the fo\-llo\-wing necessary condition: Let $h$ be a differentiable function. If
$h$ is strongly quasiconvex with modulus $\gamma>0$, then for every $x,y\in
C$, we have
\begin{equation}
h(x)\leq h(y)~\Longrightarrow~\langle \nabla h(y),y-x\rangle \geq \frac{\gamma
}{2}\lVert y-x\rVert^{2}. \label{intro}%
\end{equation}
The second part is a sufficient condition: If \eqref{intro} holds for all
$x,y\in C$, then $h$ is strongly quasiconvex with modulus $\frac{\gamma}{2}$.

Condition \eqref{intro} has been used in the last years for accelerating the
convergence of gradient-type methods with momentum and for obtaining new
examples of strongly quasiconvex functions (see
 \cite{HLMV,J-1,Lara-9,LMV}). However, the only known proof that
\eqref{intro} is sufficient for strong quasiconvexity \cite[Theorem 6]{VNC-2}
is based on a lemma the proof of which is long and tricky, and in addition it
needs amend\-ments. To be more precise, the lemma says the following:

\begin{lemma} {\rm (\cite[Lemma in page 22]{VNC-2})}
Assume that $\kappa(t)$ is a nonnegative su\-mma\-ble function that is not
everywhere zero on $\left[  0,a\right]  $, $\kappa(0)=0$, and assume that
$g:\left[  0,a\right]  \rightarrow \mathbb{R}$ is an absolutely continuous
function such that $g(0)=0\geq g(a)$. If%
\begin{equation}
g^{\prime}(u)(u-v)\geq \kappa \left(  \left \vert u-v\right \vert \right)
\left \vert u-v\right \vert \label{r:gk}%
\end{equation}
for all $u,v\in \left[  0,a\right]  $ such that $g(u)\geq g(a)$, then for all
$\lambda \in \left]  0,a\right[  $ we have%
\[
g(\lambda a)\leq-\lambda(1-\lambda)\int_{0}^{a}\kappa(t)dt.
\]
\end{lemma}

The proof starts by asserting that ``\eqref{r:gk} implies that $g$ is
quasiconvex". This is not correct, as the following example shows.

\begin{example}
Take $\alpha=4$, $\kappa(t)=0$ on $\left[  0,4\right[  $ and $\kappa(4)=1$,
$g(t)=-9+\left(  t-1\right)  ^{2}\left(  t-3\right)  ^{2}$. Note that $\kappa$
is nonnegative, summable, not everywhere zero and $\kappa(0)=0$, while $g$ is
absolutely continuous, and $g(0)=g(4)=0$. Let us check whether (\ref{r:gk}) is
satisfied: The only points $u\in \left[  0,4\right]$ such that $g(u)\geq
g(\alpha)$ are $0,4$. For $u=0$ we see that $g^{\prime}(0)=-24$. For every
$v\in \left[  0,4\right[  $ the right-hand side of (5) is $0$, while the
left-hand side is nonnegative. For $v=4$, $\kappa(4)=1$ so again (\ref{r:gk})
holds. Likewise we can check $u=4$.

Finally, note that $g$ is not quasiconvex since $g(1)=g(3)=-9$ while $g(2)=-8$.
\end{example}

The mistake in the proof lies at the point where equation (\ref{r:gk}) is used 
at a point $u$ (denoted $y^{\ast}$ in \cite{VNC-2}) that does not satisfy
$g(u) \geq g(a)$, so \eqref{r:gk} does not necessarily hold.

In this note, we present new necessary and sufficient conditions for (not
necessarily smooth) strongly quasiconvex functions. These conditions imply the
first order conditions that are based on (\ref{intro}). Furthermore, we also
improve the quadratic growth condition for strongly quasiconvex functions.

\section{Preliminaries}\label{sec:02}

We recall the definitions of the Dini derivatives, to be used in our results:
Let $I\subseteq \mathbb{R}$ be an open interval and $h:I\rightarrow \mathbb{R}$
be a function. The upper and lower Dini derivatives of $h$ at the point $s\in
I$ are defined as%
\[
h_{+}^{\prime}(s)=\limsup_{t\rightarrow0_{+}}\frac{h(s+t)-h(s)}{t},\qquad
h_{-}^{\prime}(s)=\liminf_{t\rightarrow0_{+}}\frac{h(s+t)-h(s)}{t}.
\]

If $h$ is defined on an open convex subset $C$ of $X$, then the upper and
lower Dini derivatives of $h$ at $x\in C$, in the direction $a\in X$, are
defined as%
\[
h_{+}^{\prime}(x;a)=\limsup_{t\rightarrow0_{+}}\frac{h(x+ta)-h(x)}{t},\qquad
h_{-}^{\prime}(x;a)=\liminf_{t\rightarrow0_{+}}\frac{h(x+ta)-h(x)}{t}.
\]

We will use a simple version of Saks' theorem on recovering a function
from a Dini derivative \cite[Theorem 9]{HaTh}:

\begin{theorem}
[Saks]\label{Saks}Suppose that $F$ is a continuous function defined on an
interval $I$ of $\mathbb{R}$, and $g$ is a continuous function on $I$. If
$F_{+}^{\prime}(s)\geq g(s)$ at every point $s\in I$, then%
\[
F(b)-F(a)\geq \int_{a}^{b}g(s)ds
\]
for each interval $\left[  a,b\right]  \subseteq I$.
\end{theorem}


We also recall that
a function $h$ defined on a convex set $C\subseteq X$ is called strongly 
quasiconvex with modulus $\gamma>0$ (see \cite{P}), if for every $x,y\in
C$ and $t \in [0, 1]$, we have
\begin{equation}
h(x+t(y-x))\leq \max \{h(x),h(y)\}-\frac{\gamma}{2}t(1-t)\lVert y-x\rVert^{2}.
\label{r:str-qcx}%
\end{equation}

In what follows, we will often write strong quasiconvexity in the following
equivalent manner: For every $x,y\in C$ and $z\in \lbrack x,y]$, we have%
\begin{equation}
h(x)\leq h(y)\, \Longrightarrow \,h(y)\geq h(z)+\frac{\gamma}{2}\lVert
z-x\rVert \lVert y-z\rVert. \label{r:new-form}%
\end{equation}

\section{Main Results} \label{sec:03}


Our main result, which provides necessary and sufficient conditions for strong
quasiconvexity, is given below.

\begin{theorem}\label{Th:Charact1} 
 Let $h$ be defined on a convex set $C\subseteq X$. If $h$ is strongly 
 quasiconvex with modulus $\gamma>0$, then for every $x, y \in C$
and every $z=x+t(y-x)$ with $0<t\leq1$, the following implication holds:
\begin{align}
h(x) \leq h(z) \, \Longrightarrow \, h(z)  &  \leq h(y)-\frac{\gamma}{4}(1-
t^{2}) \lVert y-x \rVert^{2}\label{no-integral}\\
&  = h(y)-\frac{\gamma}{4}(\lVert y-x\rVert^{2}-\lVert z-x\rVert^{2}).
\label{no-integral1}%
\end{align}
Conversely, if $h$ is continuous along line segments of $C$, and
\eqref{no-integral} holds for every $x,y\in C$ and every $z=x+t(y-x)$, with
$0<t\leq1$, then $h$ is strongly quasiconvex with modulus $\frac{\gamma}{2}$.
\end{theorem}

\begin{proof}
$(\Rightarrow)$: Let $h$ be strongly quasiconvex with modulus $\gamma>0$. 
Let $x, y \in C$ and assume that for some $z =x+t(y-x)$, $0<t\leq1$, we have
$h(x) \leq h(z)$. Consider a finite sequence of points $w_{i}=z+\frac{i}{n} 
(y-z)$, $i=0,1,\ldots n$. Since $z\in \left[  x,w_{1}\right]$,
(\ref{r:str-qcx}) implies that we cannot have $\max \left \{  h(x),h(w_{1}%
)\right \}  =h(x)$, thus $h(z)\leq h(w_{1})$. Using the same argument
successively for $w_{i}\in \left[w_{i-1},w_{i+1}\right]$, $i=1,\ldots, n-1$,
we find
\[
h(x)\leq h(z)=h(w_{0})\leq h(w_{1})\leq \ldots \leq h(w_{n})=h(y).
\]
Note that
\begin{equation}
\lVert w_{i}-x\rVert=\lVert z-x\rVert+\frac{i}{n}\lVert y-z\rVert
,~\forall~i\in \{0,1,\ldots,n\}, \label{wix}%
\end{equation}
because all points are on a straight line. Since $w_{i}\in \lbrack x,w_{i+1}]$,
from \eqref{wix} and \eqref{r:new-form} we deduce for $i=0,\ldots,n-1$ that
\begin{align*}
h(w_{1})  &  \geq h(w_{0})+\frac{\gamma}{2}\lVert w_{1}-w_{0}\rVert \lVert
w_{0}-x\rVert \\
&  =h(w_{0})+\frac{\gamma}{2}\frac{1}{n}\lVert y-z\rVert \left(  \lVert
z-x\rVert+\frac{0}{n}\lVert y-z\rVert \right) \\
h(w_{2})  &  \geq h(w_{1})+\frac{\gamma}{2}\lVert w_{2}-w_{1}\rVert \lVert
w_{1}-x\rVert \\
&  =h(w_{1})+\frac{\gamma}{2}\frac{1}{n}\lVert y-z\rVert \left(  \lVert
z-x\rVert+\frac{1}{n}\lVert y-z\rVert \right) \\
&  \vdots \\
h(w_{n})  &  \geq h(w_{n-1})+\frac{\gamma}{2}\lVert w_{n}-w_{n-1}\rVert \lVert
w_{n-1}-x\rVert \\
&  =h(w_{n-1})+\frac{\gamma}{2}\frac{1}{n}\lVert y-z\rVert \left(  \lVert
z-x\rVert+\frac{n-1}{n}\lVert y-z\rVert \right)  .
\end{align*}
By adding the previous inequalities, we have
\[
h(y)\geq h(z)+\frac{\gamma}{2}\lVert y-z\rVert{\displaystyle \, \sum
\limits_{i=0}^{i=n-1}}\, \frac{1}{n}\left(  \lVert z-x\rVert+\frac{i}{n}\lVert
y-z\rVert \right)  .
\]
Taking $n\rightarrow+\infty$, the sum of the right-hand side becomes a Riemann
integral
\begin{align*}
h(y)  &  \geq h(z)+\frac{\gamma}{2}\lVert y-z\rVert \int_{0}^{1}(\lVert
z-x\rVert+s\lVert y-z\rVert)ds\\
&  =h(z)+\frac{\gamma}{2}\lVert y-z\rVert \left(  \lVert z-x\rVert+\frac{1}%
{2}\lVert y-z\rVert \right) \\
&  =h(z)+\frac{\gamma}{4}(1-t^{2})\lVert y-x\rVert^{2},
\end{align*}
thus relation \eqref{no-integral} holds.

$(\Leftarrow)$: Assume that \eqref{no-integral} holds, and take $x,y\in C$
and $z=x+t(y-x)$ with $0<t\leq1$. Suppose without loss of generality that
$h(x)\leq h(y)$. Then we consider two cases. 

If $h(x)\leq h(z)$, then relation \eqref{no-integral} gives
\[
h(z)\leq h(y)-\frac{\gamma}{4}(1-t^{2})\lVert y-x\rVert^{2}\leq h(y)-\frac
{\gamma}{2}t(1-t)\lVert y-x\rVert^{2},
\]
i.e., (\ref{r:str-qcx}) holds. 

Now, assume that $h(z)<h(x)$. Since $h$ is
continuous on $[x,y]$, it has a minimum on the segment. Let
$\{w\}=\mathrm{argmin}_{[x,y]}\,h$. Suppose first that $z\in \lbrack w,y]$.
Then, by continuity, there exists $z^{\prime}\in \lbrack x,w]$ such that
$h(z^{\prime})=h(z)$ (see Figure \ref{fig:exist}). \begin{figure}[ptbh]
\centering
\includegraphics[scale=0.90]{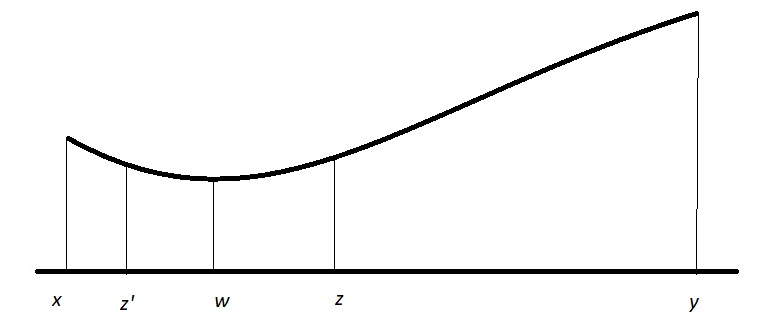}\caption{The points
$x,y,w,z,z^{\prime}$}%
\label{fig:exist}%
\end{figure}

Applying \eqref{no-integral1} to the points $z\in \lbrack z^{\prime},y]$ and
$z^{\prime}\in \lbrack x,z]$, we find
\begin{align*}
h(z)  &  \leq h(y)-\frac{\gamma}{4}(\lVert y-z^{\prime}\rVert^{2}-\lVert
z-z^{\prime}\rVert^{2})\\
h(z)  &  =h(z^{\prime})\leq h(x)-\frac{\gamma}{4}(\lVert z-x\rVert^{2}-\lVert
z-z^{\prime}\rVert^{2}).
\end{align*}
We add and use $\lVert y-z^{\prime}\rVert=\lVert y-z\rVert+\lVert z-z^{\prime
}\rVert$ and $\lVert z-x\rVert=\lVert z-z^{\prime}\rVert+\lVert z^{\prime
}-x\rVert$, which hold because all points are on the same line:
\begin{align*}
2h(z)  &  \leq2h(y)-\frac{\gamma}{4}(\lVert y-z^{\prime}\rVert^{2}+\lVert
z-x\rVert^{2}-2\lVert z-z^{\prime}\rVert^{2})\\
&  =2h(y)-\frac{\gamma}{4}(\lVert y-z\rVert^{2}+2\lVert y-z\rVert \lVert
z-z^{\prime}\rVert+\lVert z^{\prime}-x\rVert^{2}+2\lVert z^{\prime}%
-x\rVert \lVert z-z^{\prime}\rVert).
\end{align*}
The term in the parenthesis equals
\begin{align*}
&  (\lVert y-z\rVert-\lVert z^{\prime}-x\rVert)^{2}+2\lVert y-z\rVert \lVert
z^{\prime}-x\rVert+2\lVert y-z\rVert \lVert z-z^{\prime}\rVert+2\lVert
z^{\prime}-x\rVert \lVert z-z^{\prime}\rVert \\
&  =(\lVert y-z\rVert-\lVert z^{\prime}-x\rVert)^{2}+2\lVert y-z\rVert \lVert
z-x\rVert+2\lVert z^{\prime}-x\rVert \lVert z-z^{\prime}\rVert \\
&  \geq2\lVert y-z\rVert \lVert z-x\rVert.
\end{align*}
Thus,
\begin{equation}
h(z)\leq h(y)-\frac{\gamma}{4}\lVert y-z\rVert \lVert z-x\rVert=h(y)-\frac
{\gamma}{4}t(1-t)\lVert y-x\rVert^{2}. \label{gamma2}%
\end{equation}
Finally, assume that $z\in \lbrack x,w]$. Then we define $z^{\prime}\in \lbrack
w,y]$ such that $h(z^{\prime})=h(z)$ and follow the same steps as before,
interchanging $z$ with $z^{\prime}$, to arrive again at \eqref{gamma2}. It
follows that $h$ is strongly quasiconvex with modulus $\frac{\gamma}{2}$.
\end{proof}

\begin{remark}
Note that the first part of the proof does not require $h$ to be continuous or
even lower semicontinuous.
\end{remark}

As a first consequence, we have the following first-order conditions for
strong quasiconvexity of nonsmooth functions.

\begin{theorem}\label{theo:2}
 Let $h$ be a function defined on an open convex set $C\subseteq X$, 
 continuous on line segments of $C$, and $\gamma>0$. Then the following
 assertions hold:
 \begin{enumerate}
  \item[$(a)$] $h$ satisfies \eqref{no-integral} for every $x,y\in C$ and $z = 
   x + t(y-x)$, $0<t\leq1$, if and only if for every $x,y\in C$ the following 
  implication holds:
  \begin{equation}\label{r:Dini-neces}
   h(x) \leq h(y) ~ \Longrightarrow ~ h_{-}^{\prime} (y;y-x) \geq 
   \frac{\gamma}{2} \left \Vert y-x\right \Vert^{2}. 
  \end{equation}

  \item[$(b)$] If $h$ is strongly quasiconvex with modulus $\gamma>0$, then 
  for every $x, y \in C$, \eqref{r:Dini-neces} holds. Conversely, if 
  \eqref{r:Dini-neces} holds, then $h$ is strongly quasiconvex with modulus 
  $\frac{\gamma}{2}$. 
 \end{enumerate}
\end{theorem}

\begin{proof}
 $(a)$ Assume that $h$ satisfies (\ref{no-integral}) and let $x,y\in C$ be such
 that $h(x)\leq h(y)$. Take $t>0$ small enough so that $y_{t}:=y+t(y-x)\in C$.
 Using (\ref{no-integral1}) for $y\in \left[  x,y_{t}\right] $ we find
 \begin{align*}
  h(y) & \leq h(y_{t})-\frac{\gamma}{4} \left( \lVert y_{t} - x \rVert^{2} - 
  \lVert y-x \rVert^{2} \right) \\
  \Longrightarrow ~ h(y) &  \leq h(y_{t}) - \frac{\gamma}{4} ((1+t)^{2} - 1)  
  \left \Vert y-x\right \Vert ^{2} \\
  \Longrightarrow ~ \frac{h(y + t(y-x)) - h(y)}{t} &  \geq \frac{\gamma}{4} 
  (t+2) \lVert y-x \rVert ^{2}.
 \end{align*}

Taking the $\liminf$ as $t\rightarrow0_{+}$ we obtain (\ref{r:Dini-neces}).
	
Conversely, assume that (\ref{r:Dini-neces}) holds. We first show that $h$ is 
strictly quasiconvex. Indeed, assume that it is not. Then there exist $a, b 
\in C$ and $d \in \, ]a, b[$ such that $h(d) \geq \max\{h(a), h(b)\}$. Thus, 
by continuity we can find $c \in \, ]a, b[$ such that $c \in {\rm argmax}_{
[a, b]}\,h$. Since $h(c) \geq h(a)$, \eqref{r:Dini-neces} implies that 
$h^{\prime} (c; c-a)>0$. By the definition of the Dini derivative, for $t>0$
sufficiently  small we have $h(c+t(c-a))>h(c)$ and $c+t(c-a) \in [a, b]$. This 
contradicts the fact that $c \in {\rm argmax}_{[a, b]}\,h$. Thus,  $h$ is 
strictly quasiconvex.
	
	Now, let $x,y\in C$ and $z=x+t(y-x)$ with $0<t\leq1$ be such that $h(x)\leq
	h(z)$. Take any $s\in \lbrack t,1]$ and set $x_{s}=x+s(y-x)$, $g(s)=h(x_{s})$.
	Then {$z\in \,]x,x_{s}]$, }so by strict quasiconvexity, $h(x)\leq h(x_{s})$.
	Using (\ref{r:Dini-neces}) we find
	\begin{equation}
		h_{+}^{\prime}(x_{s};x_{s}-x)\geq h_{-}^{\prime}(x_{s};x_{s}-x)\geq
		\frac{\gamma}{2}\lVert x_{s}-x\rVert^{2}.\label{r:h+ineq}%
	\end{equation}
	Note that $x_{s}-x=s(y-x)$ and
	\[
	h_{+}^{\prime}(x_{s};x_{s}-x)=sh_{+}^{\prime}(x_{s};y-x)=sg_{+}^{\prime}(s).
	\]

	Thus, (\ref{r:h+ineq}) implies
	\[
	g_{+}^{\prime}(s)\geq \frac{\gamma}{2}s\lVert y-x\rVert^{2}.
	\]
	{Then, by using Theorem \ref{Saks}, }
	\begin{align*}
		g(1)-g(t) &  \geq \int_{t}^{1}\frac{\gamma}{2}s\lVert y-x\rVert^{2}ds\\
		\Longrightarrow~h(y)-h(x+t(y-x)) &  \geq \frac{\gamma}{4}\lVert y-x\rVert
		^{2}(1-t^{2}).
	\end{align*}

	Hence, $h$ satisfies \eqref{no-integral}. 
	
 $(b)$ This is an immediate consequence of part $(a)$ and Theorem  
 \ref{Th:Charact1}.
\end{proof}

\begin{remark}
 Note that in both parts of the above theorem, the continuity assumption is 
 used only for the converse.
\end{remark}

In case of smooth functions, part $(b)$ of the above theorem gives a result 
that revisits \cite[Theorems 1 and 6]{VNC-2}.

\begin{corollary}\label{Gateaux} 
 Let $h$ be G\^{a}teaux differentiable on an open convex set $C \subseteq X$. 
 If $h$ is strongly quasiconvex with modulus $\gamma>0$, then for every 
 $x, y \in C$ the following implication holds:
 \begin{equation}
  h(x) \leq h(y) ~ \Longrightarrow ~ \langle \nabla h(y), y-x \rangle \geq
  \frac{\gamma}{2} \left \Vert y-x\right \Vert ^{2}. \label{basic'}%
 \end{equation}
 Conversely, if \eqref{basic'} holds, then $h$ is strongly quasiconvex with
 modulus $\frac{\gamma}{2}$.
\end{corollary}

Another consequence of Theorem \ref{Th:Charact1} is the following 
improvement of the quadratic growth property for the unique minimizer of 
strongly quasiconvex functions. The same result was obtained in 
\cite{NamSha} for lower semicontinuous functions.

\begin{corollary}\label{strong_min} 
 Let $h$ be defined on a convex set $C \subseteq X$. If $h$ is strongly 
 quasiconvex with modulus $\gamma>0$ and $\overline{x} \in {\rm argmin}_{C}\,h$, then
 \begin{equation}
  h(\overline{x}) + \frac{\gamma}{4} \lVert y - \overline{x} \rVert^{2} \leq
  h(y), ~ \forall~ y \in C. \label{qgp:improve}%
 \end{equation}
\end{corollary}

\begin{proof}
Let $0<t\leq1$. By applying \eqref{no-integral}, we have
\[
h(\overline{x})\leq h(\overline{x}+t(y-\overline{x}))\leq h(y)-\frac{\gamma
}{4}\left(  1-t^{2}\right)  \lVert y-\overline{x}\rVert^{2}.
\]
and the result simply follows by taking $t\rightarrow0^{+}$.
\end{proof}

\medskip

\noindent \textbf{Acknowledgements} This research was partially supported 
by ANID--Chile under project Fondecyt Regular 1241040 (Lara). Also, part of 
the work was completed while the authors visited the Vietnam Institute for
Advanced study in Mathematics (VIASM), in Hanoi, Vietnam, during March and April 2025. The authors would like to thank the Institute for the hospitality.


\begin{thebibliography}{99}                                                                                               %


\bibitem{ADSZ}
 {\sc M. Avriel, W.E. Diewert, S. Schaible, I. Zang}. ``Generalized Concavity". SIAM, Classics in Applied Mathematics, Philadelphia, (2010). 


\bibitem {AE}\textsc{K.J. Arrow, A.C. Enthoven}, Quasiconcave programming,
\textit{ Econo\-me\-tri\-ca}, \textbf{29}, 779--800, (1961).

\bibitem {CamMar}\textsc{A. Cambini, L. Martein}. ``Generalized Convexity and
Optimization". Lecture Notes in Economics and Mathematical Systems no. 616,
Springer, (2009).

\bibitem {HKS}
 \textsc{N. Hadjisavvas, S. Komlosi, S. Schaible}. ``Handbook of
 Ge\-ne\-ra\-li\-zed Convexity and Generalized Monotonicity''. Springer-Verlag,
 Boston, (2005).

\bibitem {HLMV}\textsc{N. Hadjisavvas, F. Lara, R.T. Marcavillaca, P.T.
Vuong}, Heavy ball and Nesterov accelerations with Hessian-driven damping for
nonconvex optimization, arXiv: 2506.15632, (2025).

\bibitem {HaTh}\textsc{J.W. Hagood, B. S. Thomson,} Recovering a Function from
a Dini Derivative, Amer. Math. Monthly, \textbf{113}, 34--46, (2006).

\bibitem {J-1}\textsc{M. Jovanovi\'c}, Strongly quasiconvex quadratic
functions, \textit{Publ. Inst. Math., Nouv. S\'er.}, \textbf{53}, 153--156, (1993).

\bibitem {Lara-9}
\textsc{F. Lara}, On strongly quasiconvex functions:
existence results and proximal point algorithms, \textit{J. Optim. Theory
Appl.}, \textbf{192}, 891--911, (2022).

\bibitem {LMV}\textsc{F. Lara, R.T. Marcavillaca, P.T. Vuong},
Characterizations, dy\-na\-mi\-cal systems and gradient methods for strongly
quasiconvex functions, \textit{J. Optim. Theory Appl.}, Vol. 206, article
number 60, (2025).


\bibitem {NamSha}\textsc{N.M. Nam, J. Sharkansky}, On strong quasiconvexity of
functions in infinite dimensions, arxiv:2409.17450v1, (2024).

\bibitem {P}\textsc{B.T. Polyak}, Existence theorems and convergence of
minimizing sequences in extremum problems with restrictions, \textit{Soviet
Math.}, \textbf{7}, 72--75, (1966).

\bibitem {VNC-2}\textsc{A.A. Vladimirov, Ju.E. Nesterov, Ju.N. Chekanov}, On
uniformly quasi-convex functionals. (English. Russian original), \textit{Mosc. Univ.
Comput. Math. Cybern.} 1978, no 4, 19-30 (1978); translation from \textit{Vestn. Mosk. Univ.}, Ser. XV 1978, no. 4, 18-27 (1978). 
\end{thebibliography}
\end{document}